\documentclass[12pt]{amsart}
\usepackage{amssymb}
\usepackage{palatino}
\input amssym.def

\usepackage{longtable}
\usepackage{tcolorbox}

\allowdisplaybreaks

\usepackage[colorinlistoftodos]{todonotes}

\usepackage{amsmath, amsfonts}
\usepackage{amssymb}
\usepackage{amscd}
\usepackage[mathscr]{eucal}
\usepackage{palatino}
\usepackage{tabularx}
\usepackage{graphicx}
\usepackage{adjustbox}
\newfont{\cyrr}{wncyr10}

\newtheorem{thm}{Theorem}
\newtheorem{lem}{Lemma}
\newtheorem{cor}{Corollary}

\newtheorem{defi}{Definition}

\newcommand{\thmref}[1]{Theorem~\ref{#1}}

\newcommand{\lemref}[1]{Lemma~\ref{#1}}

\newcommand{\mbb}{\mathbb}

\newcommand{\psmb}{\left( \begin{smallmatrix}}
\newcommand{\psme}{ \end{smallmatrix} \right) }

\usepackage[utf8]{inputenc}

\thanks{\textit{Acknowledgment:}  The first author is supported by INSA Senior Scientist award.  The second author is supported by NBHM postdoctoral fellowship with grant number 0204/19/2017/R\& D-II/10383 at IISc Bangalore. }

\begin{document}

\title[Extension of Laguerre polynomials with negative arguments]{Extension of Laguerre polynomials with negative arguments}

\author{T. N. Shorey and Sneh Bala Sinha}

\address{T. N. Shorey \hfill Sneh Bala Sinha \newline
NIAS, Bangalore,   \hfill IISc Bangalore, \newline
560012, India.  \hfill   Department of Mathematics, \newline
$\phantom{mmmmmmmmmm}$ \hfill 560012, India. }

\email{shorey@math.iitb.ac.in} \hfill \email{24.sneh@gmail.com  ;
snehasinha@iisc.ac.in} 




\maketitle   

\begin{abstract}
We consider the irreducibility of polynomial $L_n^{(\alpha)} (x) $ where $\alpha$ is a negative integer. We observe that the constant term of $L_n^{(\alpha)} (x) $ vanishes if and only if $n \geq |\alpha| = -\alpha$. Therefore we assume that $\alpha = -n-s-1$ where $s$ is a non-negative integer. Let 
$$
g(x) = (-1)^n L_n^{(-n-s-1)}(x) = \sum\limits_{j=0}^{n}  a_j \frac{x^j}{j!} 
$$
and more general polynomial, let
$$
G(x)   =  \sum\limits_{j=0}^{n}  a_j b_j \frac{x^j}{j!} 
$$
where $b_j$ with $0 \leq j \leq n$ are integers such that $|b_0| = |b_n| = 1$. Schur was the first to prove the irreducibility of $g(x)$ for $s=0$. It has been proved that $g(x)$ is irreducibile for $0 \leq s \leq 60$. In this paper, by a different method, we prove : Apart from finitely many explicitely given posibilities, either $G(x)$ is irreducible or $G(x)$ is linear factor times irreducible polynomial. This is a consequence of the estimate $s > 1.9 k$ whenever $G(x)$ has a factor of degree $k \geq 2$ and $(n,k,s) \neq (10,5,4)$. This sharpens earlier estimates of Shorey and Tijdeman and Nair and Shorey. 
\end{abstract}

\section{\bf Introduction}
For $\alpha \in \mbb R$ and $ n \in \mbb N$,  the generalised Laguerre polynomial of degree $n$ is
\begin{equation}\label{eq_ln}
L_n^{(\alpha)} (x) = \sum\limits_{j=0}^{n} \frac{(\alpha+n) \dots (\alpha+j+1)}{ (n-j)!}\frac{(-x)^j}{j!}.
\end{equation}
These polynomials have many applications in several directions. Schur \cite{Sch} was the first to study algebraic properties of these polynomials. He gave a formula for their discriminant. He studied whether they are irreducible. Further he determined their Galois group when they are irreducible. By irreducibility of a polynomial, we shall always mean its irreducibility over the field of rational numbers. We observe that if $f$ has a factor of degree $k < n$, then $f$ has a factor of degree $n-k$. Therefore given a polynomial of degree $n$, we always consider its factor of degree $k$ with $1 \leq k \leq \frac{n}{2}$. For an account of results on the irreducibility of Laguerre polynomials, we refer to \cite{NaSh2}.

If argument $\alpha$ is a negative integer, we see that the constant term of $L_n^{(\alpha)} (x)$ vanishes if and only if $n \geq |\alpha| = -\alpha$ and then $L_n^{(\alpha)} (x)$ is reducible. Therefore we assume that $\alpha \leq -n-1$. We write $\alpha = -n-s-1$ where $s$ is a non-negative integer.  We have
\begin{equation*}
\sum\limits_{j=0}^{n} (-1)^n \frac{(n+s-j)!}{(n-j)! s!} \frac{x^j}{j!}  = L_n^{(-n-s-1)}(x).
\end{equation*}
We consider the following polynomial
\begin{align}
g(x) := g(x, n, s) = (-1)^n L_n^{(-n-s-1)}(x) & = 
\sum\limits_{j=0}^{n} \frac{(n+s-j)!}{(n-j)! s!} \frac{x^j}{j!} \nonumber\\
& = \sum\limits_{j=0}^{n}  {n+s -j \choose n-j } \frac{x^j}{j!} \nonumber\\
& =\sum\limits_{j=0}^{n}  a_j \frac{x^j}{j!} \label{new_eq_1}  
\end{align}
where $ a_j= {n+s -j \choose n-j }$ for $ 0 \leq j \leq n$.
Thus $a_n = 1$ and $  a_0 = {n+s \choose n} = \frac{(n+1) \dots (n+s)}{s!}$.
We observe that $g(x)$ is irreducible if and only if $L_n^{(-n-s-1)}(x)$ is irreducible.
We discuss the irreducibility of $g(x)$ and also of more general polynomial 
\begin{equation}\label{eq_G}
G(x) :=G(x,n,s )  =  \sum\limits_{j=0}^{n}  \pi_j \frac{x^j}{j!}  \qquad \text{ where} \quad \pi_j = b_ja_j
\end{equation}
such that $b_j \in \mbb Z$ for $0 \leq j \leq n$ with $|b_0| =1$, $|b_n| = 1$. Note that 
$|\pi_0 | = |a_0|$ and $|\pi_n| =1$. We observe that $G(x) = G(x,n,s)= g(x)$ if $b_j = 1$ for $0 \leq j \leq n$. We say that we have $(n,k,s)$ if $G(x) = G(x,n,s)$ has a factor of degree $k$ and we do not have $(n,k,s)$ if $G(x)$ has no factor of degree $k$. We write $g_1(x) = n! g(x)$ and $G_1(x) = n! G(x)$ so that $g_1$ and $G_1$ are monic polynomials with integer coefficients of degree $n$. The irreducibility of $g(x)$ and $ g_1(x)$ and also of $G(x)$ and $G_1(x)$ are equivalent. 

The irreducibilty of $g(x)$ was proved first by Schur \cite{Sch} for $s=0$, by Hajir \cite{Haj} for $s =1$, by Sell \cite{Sell} for $s =2$ and by Hajir \cite{Haj1} for $3 \leq s \leq 8$ and by Nair and Shorey \cite{NaSh1} for $9 \leq s \leq 22$. Further, developing on the work of Hajir \cite{Haj1} and Nair and Shorey \cite{NaSh1}, Jindal, Laishram and Sarma \cite{JiLaSa} extended the result of Nair and Shorey to $23 \leq s \leq 60$. Moreover, Nair and Shorey \cite{NaSh1}
proved the following theorem.

\begin{thm}\label{thm_nash}
Let $2 \leq k \leq n/2$. Let $g(x)$ be given by \eqref{new_eq_1}  with $0 \leq s \leq 1.63 k $. Then $g(x)$ has no factor of degree $k$.
\end{thm}
The proof of results in \cite{JiLaSa} depends on \thmref{thm_nash}. Further the weaker version of the above theorem was proved by Shorey and Tijdeman \cite{TijSh1} for $0 \leq s \leq 0.95k$. We generalise \thmref{thm_nash} as follows.

\begin{thm}\label{thm_ShSi}
\begin{itemize}

\item[(a)]
Let $2 \leq k \leq n/2$. Assume that $G(x)$ has a factor of degree $k \geq 2$. Then 
$s > 1.9k $ unless $(n, k, s ) \in \{ (10,5,4)\}$.

\item[(b)] Assume that the assumptions of \thmref{thm_ShSi}(a) are satisfied. Then
$$
s  >  1.99 k \quad \text{if } \quad n \geq 48732
$$
and for $0 < \phi < \frac{1}{9}$ there exists sufficiently large number $N$ depending only on $\phi$ such that 
\begin{equation}\label{eq_new_s}
s > 2(1- 2 \phi )k\hspace*{.5 cm} \text{for} \hspace*{.5 cm} n \geq N.
\end{equation}
\end{itemize}
 \end{thm}

Comparing with \thmref{thm_nash}, we observe that $g(x)$ has been replaced by more general polynomial $G(x)$ and inequality $0 \leq s \leq 1.63 k $ by $0 \leq s \leq 1.9 k$.  By checking $L_n^{(-n-s-1)}(x)$ with $(n,s) = \{(10,4) \}$  irreducible and putting $b_j =1$ for $0 \leq j \leq n$ in $G(x)$, we immediately have the following sharpening of \thmref{thm_nash}.

\begin{cor}\label{cor_shsi2}
$g(x)$ with $s \leq 1.9 k$ has no factor of degree $k \geq 2$. 
\end{cor}

\begin{thm}\label{thm_ShSi2_new}
Assume that $G(x)$ has a factor of degree $k \geq 2$. Then $s > 92$ unless
\begin{align*}
(n,k,s) \in  \{& (4,2,7),(4,2,23), (9,2,19), (9,2,47),  (16,2,14),(16,2,34), (16,2,89), (9,3,47),\\ &  (16, 3, 19),   (10,5,4) \nonumber\}.
\end{align*} 
\end{thm}
Thus, except for an explicitly given finite set of $(n,k,s)$ in \thmref{thm_ShSi2_new}, we see that $G(x)$ is either irreducible or linear factor times of an irreducible polynomial whenever $s \leq 92$.
Schur proved the irreducibility of $G(x)$ when $s=0$. Thus $\sum\limits_{j=0}^{n}b_j \frac{x^j}{j!}$ is irreducible. Therefore we always assume that $s > 0$ in this paper. \thmref{thm_ShSi2_new} may be viewed as an extension of Schur result.  Like $g(x)$, the results of Hajir on Newton polygons are not applicable for $G(x)$ as the coefficients $\pi_j$'s are not fixed. Therefore the proofs of our theorems are different from the proofs given in \cite{Haj1}, \cite{NaSh1} and \cite{JiLaSa}. Infact they follow the lines for the proofs of the results in \cite{TijSh1}. Assume that $G(x)$ has a factor of degree $k \geq 2$. First we show that $s > 9$ unless
\begin{equation}\label{new_eq_s>9}
(n,k,s) \in \{ (4,2,7), (10, 5, 4) \},
\end{equation}
see \lemref{lem_s>9}. If $s  \leq 1.9k$ and $(n,k,s) \neq (10,5,4)$, we see that $k \geq 6$ which we show is not possible by \lemref{lem_1} and \lemref{lem_3}. Therefore $s > 1.9 k$ as in \thmref{thm_ShSi}. Further we derive from $s > 1.9k$ that $s > 92$ as in \thmref{thm_ShSi2_new}. The proofs depend on combining \lemref{lem_1} and \lemref{lem_3} at several instances as explained completely in the proof of \lemref{lem_s>9_n>4k}. Also several results on lower bounds for the greatest prime factor of consecutive integers in arithmetic progression have been applied. The limit $92$ for $s$ in \thmref{thm_ShSi2_new} is optimal in the sense that the proofs depend on the Tables of Najman \cite{LN} and we should enlarge the Tables for relaxing the limit of $s$ in \thmref{thm_ShSi2_new}. This will be computationally very difficult. The results proved in this paper are general. \thmref{thm_ShSi2_new} is derived from  \thmref{thm_ShSi} which is general.

Suppose that $G(x)$ has a factor of degree $1$ and $s \geq 2$. Then, by \lemref{lem_1},  $p$ divides $\frac{(n+1) \dots (n+s)}{s!}$ for every prime $p$ dividing $n$. We have

\begin{lem}\label{lem_k_eq_1}
Assume that $G(x) $ has a factor of degree $1$. Then 
$$
n \leq s^{\pi(s)} \leq e^{s\Big( 1+ \frac{1.2762}{\log s} \Big)} \quad \text{for} \quad s \geq 2.
$$
\end{lem}

\begin{proof}
By \lemref{lem_1}, we have 
\begin{equation}
p^{v_p(n)} \leq s
\end{equation}
otherwise
$$
v_p\Big( \frac{(n+1) \dots (n+s)}{s!} \Big) \leq v_p \Big(\frac{s!}{s!}  \Big) \leq 0.
$$
This is a contradiction. Therefore 
$$ n=\prod_{p|n} p^{v_p(n)} \leq \prod_{p|n} p^{\big( \frac{\log s}{\log p} \big)} = s^{\pi(s)}. $$
Hence $n \leq s^{s\Big( 1+ \frac{1.2762}{\log s} \Big)} $ for $ s \geq 2$.
\end{proof}

The computations in this paper are carried out by SAGE.

\section{\bf{Lemmas}}
We always use $p$ for a prime and $n, s$ for integers with $n \geq 1$ and  $s \geq 0$   unless otherwise specified. We denote $x(x-1)\dots (x-k+1) = {}_k(x)$. In this section, we state lemmas required for the proof of our theorems.

The following lemma is proved by Shorey and Tijdeman \cite[Lemma 4.2]{TijSh1} in 2010.

\begin{lem}\label{lem_1}
Let $a_0, \dots, a_n$ be arbitrary integers and $h(x)= \sum\limits_{j=0}^{n}  a_j \frac{x^j}{j!}  $. Assume that $h(x) $ has a factor of degree $k \geq 1$. Suppose that there exists a prime $p > k+a$ such that $p$ divides $_{k}(n+a)$. Then $p$ divides $a_0a_n$. 
\end{lem}

We shall always use \lemref{lem_1} for $a=0$. In 2004, Laishram and Shorey  \cite[Theorem 1]{LaSh_new} proved 

\begin{lem}\label{new_lem_lash4}
If $k \geq 2$ and $n >k$, then 
$$
P(n(n+2)\dots (n+(k-1)2)) > 2k.
$$
\end{lem}

In 2005, Laishram and Shorey  \cite[Theorem 1(a)]{LaSh} proved 

\begin{lem}\label{lem_max_n}
If $k \geq 2$ and $n \geq \max\Big(2k+13, \frac{541}{262}k \Big)$, then 
$$
P({}_k(n)) > 2k.
$$
\end{lem}

In 2016, Nair and Shorey \cite[Theorem 1]{NaSh} proved the following sharpening of a result of Sylvester \cite{Syl} that a product of $k \geq 2$ consecutive positive integers each exceeding $k$ is divisible by a prime greater than $k$.
 \begin{lem}\label{new_lem_1}
$n > 4k$ and $k \geq 2$ be such that $(n,k) \not\in T$. Then 
$P({}_{k}(n) ) > 4.42 k$ and moreover  $P({}_{k}(n) )  > 4.5 k$ unless $k \not\in \{82, 83\}$. Here 
\begin{align*}
T = & \{ (9,2),(14,2),(15,2),(20,2),(24,2),(27,2),(35,2),(48,2),(49,2),(63,2),(80,2),(125,2),\\ & (224,2),(2400,2),(4374,2),(13,3),(14,3),(20,3),(24,3),(25,3),(26,3),(48,3), (54,3),\\ &  (63,3), (64,3),(98,3),(350,3),(24,4),(25,4),(32,4),(33,4),(48,4),(49,4),(63,4), (24,5), \\ & (32,5),  (48,5),(29,7),(30,7)\}.
\end{align*}
\end{lem}

\begin{lem}\label{lem_new_4}
We have 
$$
\{ x : P({}_9(x)) < 100 \} = \{ 292 \}.
$$
This is due to Luca and Najman \cite[Corollary 6]{LN}. Thus 
$$
\{ x : P({}_k(x)) < 100 \} = \emptyset \quad \text{for} \quad k> 9.
$$
\end{lem}

\begin{defi}[Newton Polygon]\label{Def_newton}
For a prime $q$ and non zero rational numbers  $\frac{a}{b}$ with $(a,b ) = 1$ and $ab \neq 0$, the $q$-adic valuation is given by
$$
 \nu_q \Big(\frac{a}{b} \Big) = j_1 -j_2 \quad \text{where} \quad
q^{j_1} ||a,~~~~ q^{j_2} || b.
$$
We define $\nu_q(0) = + \infty$. Let $h(x) = \sum\limits_{i=0}^{n}a_i x^i \in \mbb Z[x] $ with $a_0a_n \neq 0$. Let 
$$
T = \{ (0, \nu_q(a_n)), (1, \nu_q(a_{n-1})), \dots, (n, \nu_q(a_0))   \}
$$
be a set of points in the extended plane $\mbb R \cup \{\infty\}$. We consider the lower edges along the convex hull of these points. The left most edge has initial point $(0, \nu_q(a_n))$ and the right most edge has end point $(n, \nu_q(a_0))$.  The end points of all edges belong to $T$. The slopes of the edges are in the increasing order when we calculate from the left to the right. The polygon path formed by joining these edges is called the Newton polygon for $h(x)$ with respect to $q$. A lattice point of the edge is an integer point on the edge other than the end points of the edge.
\end{defi}

In 1995, Filaseta \cite[Lemma 2]{Fila} gave the following lemma.

\begin{lem}[Filaseta]\label{lem_3}
Let $k$ and $l$ be integers with $k > l \geq 0$. Suppose $h(x) = \sum\limits_{j = 0}^{n}b_jx^j \in \mbb Z[x]$ and $p$ prime such that $p \nmid b_n$ and $p$ divides $b_j$ for $j$ in $\{ 0,1, \dots , n-l-1 \}$ and the right most edge of the Newton polygon for $h(x)$ with respect to $p$ has the slope less than $1/k$. Then for any integers $a_0 , a_1, \dots , a_n$ with $|a_0| = |a_n| =1$, the polynomial $f(x) = \sum\limits_{j = 0}^{n} a_j b_jx^j \in \mbb Z[x]$ can not have a factor with degree in the interval $[l+1, k]$.
\end{lem}

In 1906, G. Dumus \cite[Lemma 2.1]{Dum} gave the following result.
\begin{lem}\label{lem_dumus}
Let $f(x)$ and $h(x)$ be in $ \mbb Z[z]$ with $f(0)h(0) \neq 0$,  and let $p$ be a prime. Let $r$
be a non-negative integer such that $p^r$ divides the leading coefficient of $f(x)h(x)$ but $p^{r+1}$ does not. Then the edges of Newton polygon for $f(x)h(x)$ with respect to $p$ can be formed by constructing a polygonal path beginning at $(0,r)$ and using translates of the edges in the Newton polygon for $f(x)$ and $h(x)$ with respect to the prime $p$. Necessarily, the translated edges are translated in such a way as to form a polygonal path with the slope of the edges increasing.
\end{lem}

For \lemref{lem_prime},  see \cite{HeiKem}. 

\begin{lem}\label{lem_prime}
\begin{itemize}
\item[(a)] Let $\theta \in \{ \frac{1}{39},  \frac{1}{1000}\}$ and $m_0(\theta)= 800$ and $48683$ according as $\theta  = \frac{1}{39}$ and $\frac{1}{1000}$, respectively. Then there exists a prime $p$ satisfying
$$
m < p < (1+ \theta) m \quad \text{for} \quad m \geq m_0(\theta).
$$
\item[(b)]   $\pi(1.064286 m) - \pi (m) >0 $ for $m \geq 140$.\label{lem_pnt}
\end{itemize}
\end{lem}

\begin{lem}\label{lem_ShSi}
For $\theta > 0$, assume that there exists a prime in $(m, (1+ \theta) m)$ whenever  $m \geq m_0= m_0(\theta)$ where $m_0$ is depending only on $\theta$ and $m_0(\theta)$ is given by \lemref{lem_prime}(a) when $\theta \in  \{ \frac{1}{39}, \frac{1}{1000}\}$. Assume that $G(x)$ has a factor of degree $k \geq 2$. Let $2k \leq n \leq 4k$ and $0 < \phi < \frac{1}{9}$. Then
$$
s > 2(1-2\phi) k \quad \text{for} \quad n \geq \frac{m_0 (\frac{\phi}{1-\phi})}{1-\phi} . 
$$
\end{lem}

\begin{proof}
Let $0 < \phi < \frac{1}{9}$. We consider the interval $(n-\phi n, n)$. Put
$$
n- \phi n = (1- \phi)n = m.
$$
Then $$ (n-\phi n, n) = \Big( m, \frac{m}{1- \phi}  \Big) = (m, (1+\theta)m)$$
where $\theta = \frac{\phi}{1-\phi}  $.
Therefore we see from \lemref{lem_prime} (a) that there exists $m_0(\theta)$ such that for $m \geq m_0(\theta)$, we have

$$
p_0 \in (m, (1+\theta)m) = (n-\phi n, n)
$$
where $p_0$ is a prime. 

We may assume
$$
n \geq \frac{m_0 ( \frac{\phi}{1-\phi} )}{1-\phi} .
$$
Then $m \geq m_0(\theta)$ since $n = \frac{m}{1-\phi}$ and $\theta = \frac{\phi}{1-\phi}$.
Now $\phi n < \frac{4k}{9}< k-1$ by $n \leq 4k$ and $k \geq 2$. Therefore we observe that  $(n-\phi n, n)  \subseteq (n-k+1, n)$. By \lemref{lem_1} there exists $0 \leq i \leq \phi n$ such that $p_0 | (n-i)$ and $p_0 | \frac{{}_s(n+s)}{s!}$. Further $p_0 > (1 - \phi) n$ and we can assume that $(1 - \phi) n > s$ otherwise
$$
2(1-2\phi) k  < (1-\phi)n \leq s
$$
 since $n \geq 2k$ and the assertion follows. Therefore there exists $j$ with $1 \leq j \leq s$ such that $p_0 | (n+j)$. Thus
$$
(1 - \phi) n < p_0 \leq i+j \leq \phi n +s
$$
which implies that $s > (1-2\phi) n \geq 2(1-2\phi) k $.
\end{proof}

\begin{lem}\label{lem_s>p^r}
Assume that $G(x)$ has a factor of degree $k= 2$. Let $\delta \in \{  0,1 \}$ and $r> 0$ be an integer. Let $p$ be a prime such that $p > k$ and $p^r | (n-\delta) $. Let $(s+1, p) = 1$ if $\delta = 1$. Then
$$
s \geq p^r - \delta.
$$ 
\end{lem}

\begin{proof}
Let $s < p^r -\delta$. Let $\delta = 0$. Then $s < p^r$. We consider
\begin{align*}
\nu_p\big( (n+1)(n+2) \dots (n+s)\big) & = \Big[{\frac{s}{p}} \Big]+ \dots + \Big[{\frac{s}{p^{r-1}}}\Big] \\ & = \nu_p(s!).
\end{align*}
Let $\delta = 1$. Then $s < p^r-1$. We consider
\begin{align*}
\nu_p\big( (n+1)(n+2) \dots (n+s)\big) & = \nu_p\big( (n-1+2)(n-1+3) \dots (n-1+s+1)\big)\\ & 
= \nu_p\big((n-1+1) (n-1+2)(n-1+3) \dots (n-1+s+1)\big) \\ & 
= \Big[{\frac{s+1}{p}}\Big] + \dots + \Big[{\frac{s+1}{p^{r-1}}}\Big] \\ & = \nu_p((s+1)!).
\end{align*}

When $\delta =0$, we have
\begin{align*}
\nu_p\Big(\frac{(n+1)(n+2) \dots (n+s)}{s!} \Big) & = \nu_p\big( (n+1)(n+2) \dots (n+s)\big) - \nu_p(s!) \\ & = 0.
\end{align*}

When $\delta =1$, we have
\begin{align*}
\nu_p\Big(\frac{(n+1)(n+2) \dots (n+s)}{s!} \Big) & = \nu_p((s+1)!) - \nu_p(s!) \\ & = \nu_p((s+1))  
\\ & = 0  \quad \text{since} \quad (s+1, p) =1.
\end{align*}
This gives a contradiction to \lemref{lem_1}.
\end{proof}

\begin{lem}\label{lem_dum_exp}
The polynomial $G(x) = G(x,n,s)$ has no factor of degree $k$ when $(n,k,s)$ belongs to the following set.
\begin{align*}
\Omega = \{& (4,2,2),(4,2,47),(6,2,4), (6,2,14), (6,2,79),(8,2,6),(8,2,13),(8,2,48),(8,2,62),\\ & (9,2,26), (12,2,43), (16,2,19),(16,2,24), (16,2,79),   (18,2,16),  (24,2,22),  (32,2,30),  \nonumber\\ &(48,2,46),   (54,2,52),
(64,2,62),  (72,2,70),  (6,3,47), (8,3,6),(8,3,13), (9,3,5), \nonumber\\ & (9,3,6),(9,3,13),  (16,3,12), (16,3,13),(16,3,62),  (18,3,33),
(36,3,15),   (9,4,5),  \nonumber\\ &  (9,4,6), (16,4,12),  (16,4,62), (28,4,24),(81,4,77),(10,5,5)  \nonumber\}.
\end{align*}
\end{lem}

\begin{proof}
Let $g_1(x) = n! g(x)$ where $g(x)$ defined in \eqref{new_eq_1} and $G_1(x) = n! G(x) $ where $G(x)$ is defined in \eqref{eq_G}.  We observe that the left most and the right most vertices of $g_1(x)$  and $G_1(x)$ are same since $|b_0| = |b_n| = 1$. 
By \lemref{lem_dumus}, we show that $G_1(x)$ does not have a factor of degree $k$ for all  $(n,k,s)$ given in $\Omega$.  We prove it for $(6,3,47)$ and $(10,5,5)$ and the details for the proofs of others in $\Omega$are similar. We are providing $(n,k,s,p)$ where $p$ is a prime number with respect to which \lemref{lem_dumus} has been applied for the Newton polygon of $G_1(x)$ and a complete explanation for two triples.
\begin{align}\label{new_eq_n<301}
\{& (4,2,2,2),(4,2,47,2),(6,2,4,5), (6,2,14,2), (6,2,79,5),(8,2,6,2),(8,2,13,7),(8,2,48,2), \\ &(8,2,62,2),  (9,2,26,3),(12,2,43,11), (16,2,19,5),(16,2,24,5), (16,2,79,5),   (18,2,16,17),  \nonumber\\ & (24,2,22,23), (32,2,30,31),  (48,2,46,47),   (54,2,52,53),
(64,2,62,7),  (72,2,70,71),  \nonumber\\ & (6,3,47,2), (8,3,6,2), (8,3,13,7), (9,3,5,2),(9,3,6,3),(9,3,13,3), (16,3,12,2), (16,3,13,2),  \nonumber\\ & (16,3,62,2),  (18,3,33,3),
(36,3,15,3),   (9,4,5,3), (9,4,6,3), (16,4,12,2),  (16,4,62,2),\nonumber\\ & (28,4,24,3),(81,4,77,79),(10,5,5,3)  \nonumber\}.
\end{align}

In the case $(n,k,s) = (6,3,47)$, we consider the Newton polygons of $g_1(x)$ and $G_1(x)$ with respect to a prime $2$. The Newton polygon of $g_1(x)$ is a single edge joining $(0,0)$ to $(6,7)$ with no lattice point. Therefore $g_1(x)$ is irreducible and the Newton polygon of $G_1(x)$ coincides with the Newton polygon of  $g_1(x)$. Hence $G_1(x)$ is irreducible and has no factor of degree $3$.

For the case $(n,k,s) = (10,5,5)$, we consider the Newton polygons of $g_1(x)$ and $G_1(x)$ with respect to a prime $3$. The Newton polygon of $g_1(x)$ consists of two edges, one joining $(0,0)$ to $ (9,4)$  and other joining $ (9,4)$ to $(10,5)$. Therefore  $g_1(x)$ has no factor of degree $5$ by \lemref{lem_dumus}. Let $3 | b_9$. Then the Newton polygons of $G_1(x)$ is a single edge joining $(0,0)$ to $ (10,5)$ having lattice points $(2,1), (4,2), (6,3)$ and $ (8,4)$. Therefore we may suppose that $3 \nmid b_9$.  Then the Newton polygon of $G_1(x)$ coincides with the Newton polygon of  $g_1(x)$ and hence it has no factor of degree $5$ by \lemref{lem_dumus}.
\end{proof}

\begin{lem}\label{lem_s>9_n>4k}
Let $n > 4k$ with $k \geq 2$ and $(n,k) \in T$ where $T$ is given by \lemref{new_lem_1}. Assume that $G(x)$ has a factor of degree $k$. Then $s > 9$.
\end{lem}

\begin{proof}
Let $n > 4k$. Assume that $s \leq 9$.  Suppose $G(x)$, as defined  in \eqref{eq_G}, has a factor of degree $k \geq 2$ and we shall arrive at a contradiction.  Let $(n,k) \in T$ where $T$ is given by \lemref{new_lem_1}. Suppose $k =2$. For showing $G(x)$ has no factor of degree $2$, we need to show the existence of a prime $p_0$ such that $p_0 | {}_{k}(n)$ with $p_0$ does not divide $\frac{{}_s(n+s)}{s!}$ by \lemref{lem_1}. The following Table \ref{tab; table_k=2} is the list of all $(n,s) \in T$ which do not satisfy \lemref{lem_1} together with a prime $p$ by which \lemref{lem_1} has been applied. 

\begin{center}
\begin{longtable}{|c|c|c||c|c|c||c|c|c||c|c|c||c|c|c||c|c|c|}
\caption{For pairs in $T$ with $k =2$ not satisfying \lemref{lem_1}.}
 \label{tab; table_k=2}\\
\hline
$n$ & $s$ &$p$ & $n$ & $s$ &$p$& $n$ & $s$ & $p$ & $n$ & $s$ & $p$ & $n$ & $s$ & $p$& $n$ & $s$ & $p$\\
\hline
\endfirsthead
\multicolumn{18}{c}%
{\tablename\ \thetable\ -- \textit{For pairs in $T$ with $k =2$ not satisfying \lemref{lem_1} } }\\
\hline
$n$ & $s$ &$p$ & $n$ & $s$ &$p$& $n$ & $s$ & $p$ & $n$ & $s$ & $p$ & $n$ & $s$ & $p$& $n$ & $s$ & $p$\\
\hline
\endhead
\hline \multicolumn{18}{r}{\textit{For pairs in $T$ with $k =2$ not satisfying \lemref{lem_1}} }\\
\endfoot
\hline
\endlastfoot
$9$ & $1$ & $3$& $9$ & $2$ & $3$& $9$ & $3$ & $3$& $9$ & $4$ & $3$& $9$ & $5$ & $3$& $9$ & $6$ & $3$\\ \hline
$9$ & $7$ & $3$& $9$ & $8$ & $3$& $9$ & $9$ & $3$& $14$ & $1$ & $13$& $14$ & $2$ & $13$& $14$ & $3$ & $13$\\ \hline
$14$ & $4$ & $13$& $14$ & $5$ & $13$& $14$ & $6$ & $13$& $14$ & $7$ & $13$& $14$ & $8$ & $13$& $14$ & $9$ & $13$\\ \hline
$15$ & $1$ & $7$& $15$ & $2$ & $7$& $15$ & $3$ & $7$& $15$ & $4$ & $7$& $15$ & $5$ & $7$& $15$ & $6$ & $5$\\ \hline
$15$ & $7$ & $7$& $15$ & $8$ & $7$& $15$ & $9$ & $7$& $20$ & $1$ & $19$& $20$ & $2$ & $19$& $20$ & $3$ & $19$\\ \hline
$20$ & $4$ & $19$& $20$ & $5$ & $19$& $20$ & $6$ & $19$& $20$ & $7$ & $19$& $20$ & $8$ & $19$& $20$ & $9$ & $19$\\ \hline
$24$ & $1$ & $23$& $24$ & $2$ & $23$& $24$ & $3$ & $23$& $24$ & $4$ & $23$& $24$ & $5$ & $23$& $24$ & $6$ & $23$\\ \hline
$24$ & $7$ & $23$& $24$ & $8$ & $23$& $24$ & $9$ & $23$& $27$ & $1$ & $13$& $27$ & $2$ & $13$& $27$ & $3$ & $13$\\ \hline
$27$ & $4$ & $13$& $27$ & $5$ & $13$& $27$ & $6$ & $13$& $27$ & $7$ & $13$& $27$ & $8$ & $13$& $27$ & $9$ & $13$\\ \hline
$35$ & $1$ & $17$& $35$ & $2$ & $17$& $35$ & $3$ & $17$& $35$ & $4$ & $17$& $35$ & $5$ & $17$& $35$ & $6$ & $17$\\ \hline
$35$ & $7$ & $17$& $35$ & $8$ & $17$& $35$ & $9$ & $17$& $48$ & $1$ & $47$& $48$ & $2$ & $47$& $48$ & $3$ & $47$\\ \hline
$48$ & $4$ & $47$& $48$ & $5$ & $47$& $48$ & $6$ & $47$& $48$ & $7$ & $47$& $48$ & $8$ & $47$& $48$ & $9$ & $47$\\ \hline
$49$ & $1$ & $7$& $49$ & $2$ & $7$& $49$ & $3$ & $7$& $49$ & $4$ & $7$& $49$ & $5$ & $7$& $49$ & $6$ & $7$\\ \hline
$49$ & $7$ & $7$& $49$ & $8$ & $7$& $49$ & $9$ & $7$& $63$ & $1$ & $31$& $63$ & $2$ & $31$& $63$ & $3$ & $31$\\ \hline
$63$ & $4$ & $31$& $63$ & $5$ & $31$& $63$ & $6$ & $31$& $63$ & $7$ & $31$& $63$ & $8$ & $31$& $63$ & $9$ & $31$\\ \hline
$80$ & $1$ & $79$& $80$ & $2$ & $79$& $80$ & $3$ & $79$& $80$ & $4$ & $79$& $80$ & $5$ & $79$& $80$ & $6$ & $79$\\ \hline
$80$ & $7$ & $79$& $80$ & $8$ & $79$& $80$ & $9$ & $79$& $125$ & $1$ & $31$& $125$ & $2$ & $31$& $125$ & $3$ & $31$\\ \hline
$125$ & $4$ & $31$& $125$ & $5$ & $31$& $125$ & $6$ & $31$& $125$ & $7$ & $31$& $125$ & $8$ & $31$& $125$ & $9$ & $31$\\ \hline
$224$ & $1$ & $223$& $224$ & $2$ & $223$& $224$ & $3$ & $223$& $224$ & $4$ & $223$& $224$ & $5$ & $223$& $224$ & $6$ & $223$\\ \hline
$224$ & $7$ & $223$& $224$ & $8$ & $223$& $224$ & $9$ & $223$& $2400$ & $1$ & $2399$& $2400$ & $2$ & $2399$& $2400$ & $3$ & $2399$\\ \hline
$2400$ & $4$ & $2399$& $2400$ & $5$ & $2399$& $2400$ & $6$ & $2399$& $2400$ & $7$ & $2399$& $2400$ & $8$ & $2399$& $2400$ & $9$ & $2399$\\ \hline
$4374$ & $1$ & $4373$& $4374$ & $2$ & $4373$& $4374$ & $3$ & $4373$& $4374$ & $4$ & $4373$& $4374$ & $5$ & $4373$& $4374$ & $6$ & $4373$\\ \hline
$4374$ & $7$ & $4373$& $4374$ & $8$ & $4373$& $4374$ & $9$ & $4373$&&&&&&&&&
\end{longtable}
\end{center}

Denote by $S$ the set of all pairs $(n,s) $ with $1 \leq s \leq 9$ and $ (n,2) \in T$ satisfying \lemref{lem_1}. Then
\begin{align*}
S = \{ & (9,3), (9,4),(9,5), (9,6), (9,7), (9,8), (9,9),(15,6), (15,7),(15,8),(15,9),(49,7),\\ &(49,8),(49,9)\}.
\end{align*}
We consider 
\begin{equation}\label{eq_cj}
g_1(x) = n!g(x) = \sum\limits_{j=0}^{n}  \frac{n!}{j!} a_j x^j = \sum\limits_{j=0}^{n}  c_j x^j  \quad \text{where} \quad c_j = \frac{n!}{j!} a_j \in \mbb Z.
\end{equation}
We check that \lemref{lem_3} holds with $h(x) = g_1(x)$ and $l=1$  and a suitable prime $p$ for each pair in $S$. This contradicts that $G(x)$ has a factor of degree $k=2$. The value of $p$ for each pair in the set $S$ satisfying \lemref{lem_3}  is given by the following Table \ref{my-label}.

\smallskip
\begin{table}[!h]
  \centering
  \caption{Each pair in $S$ is satisfying \lemref{lem_3}}
  \label{my-label}
 \begin{adjustbox}{width=\textwidth}
\begin{tabular}{|c|c|c|}
\hline 
 $n$ & $s$ & $p$ \\ 
\hline 
$9$ & $3$ & $3$ \\ 
\hline 
$9$ & $4$ & $3$ \\ 
\hline 
\end{tabular}
\begin{tabular}{|c|c|c|}
\hline 
 $n$ & $s$ & $p$ \\ 
\hline 
$9$ & $5$ & $3$ \\ 
\hline 
 $9$ & $6$ & $3$ \\ 
\hline 
\end{tabular}
\begin{tabular}{|c|c|c|}
\hline 
 $n$ & $s$ & $p$ \\ 
\hline 
 $9$ & $7$ & $3$ \\ 
\hline 
 $9$ & $8$ & $3$ \\ 
\hline 
\end{tabular}
\begin{tabular}{|c|c|c|}
\hline 
 $n$ & $s$ & $p$ \\ 
\hline 
 $9$ & $9$ & $3$ \\ 
\hline 
 $15$ & $6$ & $5$\\ 
\hline   
\end{tabular}
\begin{tabular}{|c|c|c|}
\hline 
 $n$ & $s$ & $p$ \\ 
\hline 
$15$ & $7$ & $5$\\ 
\hline 
 $15$ & $8$ & $5$\\ 
\hline 
\end{tabular}
\begin{tabular}{|c|c|c|c|}
\hline 
 $n$ & $s$ & $p$ \\ 
\hline 
$15$ & $9$ & $3$ \\ 
\hline
$49$ & $7$ & $7$\\ 
\hline   
\end{tabular}
\begin{tabular}{|c|c|c|c|}
\hline 
 $n$ & $s$ & $p$ \\ 
\hline 
 $49$ & $8$ & $5$\\ 
\hline 
 $49$ & $9$ & $5$ \\ 
\hline  
\end{tabular}
\end{adjustbox}
\end{table}

Let $k >2$. We check that all the pairs in $T$  do not satisfy \lemref{lem_1}, see Table \ref{tab;table_2.1} where the value of $p$ with which \lemref{lem_1} has been applied is given. Now the assertion of \lemref{lem_s>9_n>4k} follows immediately.

\begin{center}
\begin{longtable}{|c|c|c|c||c|c|c|c||c|c|c|c||c|c|c|c||c|c|c|c||c|c|c|c|}
\caption{}
 \label{tab;table_2.1}\\
\hline
$n$ &$k $&  $s$ & $p$ & $n$ & $k $&  $s$ & $p$ & $n$ & $k $&  $s$ & $p$ & $n$ & $k $&  $s$ & $p$ & $n$ & $k $ &  $s$ & $p$& $n$ & $k $&  $s$ & $p$\\
\hline
\endfirsthead
\multicolumn{24}{c}%
{\tablename\ \thetable\ } \\
\hline
$n$ &$k $&  $s$ & $p$ & $n$ & $k $&  $s$ & $p$ & $n$ & $k $&  $s$ & $p$ & $n$ & $k $&  $s$ & $p$ & $n$ & $k $ &  $s$ & $p$& $n$ & $k $&  $s$ & $p$\\
\hline
\endhead
\hline \multicolumn{24}{r}{\textit{}} \\
\endfoot
\hline
\endlastfoot
$13$ & $3$ & $1$ & $11$& $13$ & $3$ & $2$ & $11$& $13$ & $3$ & $3$ & $11$& $13$ & $3$ & $4$ & $11$& $13$ & $3$ & $5$ & $11$& $13$ & $3$ & $6$ & $11$ \\ \hline
$13$ & $3$ & $7$ & $11$& $13$ & $3$ & $8$ & $11$& $13$ & $3$ & $9$ & $13$& $14$ & $3$ & $1$ & $7$& $14$ & $3$ & $2$ & $7$& $14$ & $3$ & $3$ & $7$ \\ \hline
$14$ & $3$ & $4$ & $7$& $14$ & $3$ & $5$ & $7$& $14$ & $3$ & $6$ & $7$& $14$ & $3$ & $7$ & $13$& $14$ & $3$ & $8$ & $13$& $14$ & $3$ & $9$ & $13$ \\ \hline
$20$ & $3$ & $1$ & $5$& $20$ & $3$ & $2$ & $5$& $20$ & $3$ & $3$ & $5$& $20$ & $3$ & $4$ & $5$& $20$ & $3$ & $5$ & $19$& $20$ & $3$ & $6$ & $19$ \\ \hline
$20$ & $3$ & $7$ & $19$& $20$ & $3$ & $8$ & $19$& $20$ & $3$ & $9$ & $19$& $24$ & $3$ & $1$ & $11$& $24$ & $3$ & $2$ & $11$& $24$ & $3$ & $3$ & $11$ \\ \hline
$24$ & $3$ & $4$ & $11$& $24$ & $3$ & $5$ & $11$& $24$ & $3$ & $6$ & $11$& $24$ & $3$ & $7$ & $11$& $24$ & $3$ & $8$ & $11$& $24$ & $3$ & $9$ & $23$ \\ \hline
$25$ & $3$ & $1$ & $5$& $25$ & $3$ & $2$ & $5$& $25$ & $3$ & $3$ & $5$& $25$ & $3$ & $4$ & $5$& $25$ & $3$ & $5$ & $23$& $25$ & $3$ & $6$ & $23$ \\ \hline
$25$ & $3$ & $7$ & $23$& $25$ & $3$ & $8$ & $23$& $25$ & $3$ & $9$ & $23$& $26$ & $3$ & $1$ & $5$& $26$ & $3$ & $2$ & $5$& $26$ & $3$ & $3$ & $5$ \\ \hline
$26$ & $3$ & $4$ & $13$& $26$ & $3$ & $5$ & $13$& $26$ & $3$ & $6$ & $13$& $26$ & $3$ & $7$ & $13$& $26$ & $3$ & $8$ & $13$& $26$ & $3$ & $9$ & $13$ \\ \hline
$48$ & $3$ & $1$ & $23$& $48$ & $3$ & $2$ & $23$& $48$ & $3$ & $3$ & $23$& $48$ & $3$ & $4$ & $23$& $48$ & $3$ & $5$ & $23$& $48$ & $3$ & $6$ & $23$ \\ \hline
$48$ & $3$ & $7$ & $23$& $48$ & $3$ & $8$ & $23$& $48$ & $3$ & $9$ & $23$& $54$ & $3$ & $1$ & $13$& $54$ & $3$ & $2$ & $13$& $54$ & $3$ & $3$ & $13$ \\ \hline
$54$ & $3$ & $4$ & $13$& $54$ & $3$ & $5$ & $13$& $54$ & $3$ & $6$ & $13$& $54$ & $3$ & $7$ & $13$& $54$ & $3$ & $8$ & $13$& $54$ & $3$ & $9$ & $13$ \\ \hline
$63$ & $3$ & $1$ & $7$& $63$ & $3$ & $2$ & $7$& $63$ & $3$ & $3$ & $7$& $63$ & $3$ & $4$ & $7$& $63$ & $3$ & $5$ & $7$& $63$ & $3$ & $6$ & $7$ \\ \hline
$63$ & $3$ & $7$ & $31$& $63$ & $3$ & $8$ & $31$& $63$ & $3$ & $9$ & $31$& $64$ & $3$ & $1$ & $7$& $64$ & $3$ & $2$ & $7$& $64$ & $3$ & $3$ & $7$ \\ \hline
$64$ & $3$ & $4$ & $7$& $64$ & $3$ & $5$ & $7$& $64$ & $3$ & $6$ & $31$& $64$ & $3$ & $7$ & $31$& $64$ & $3$ & $8$ & $31$& $64$ & $3$ & $9$ & $31$ \\ \hline
$98$ & $3$ & $1$ & $7$& $98$ & $3$ & $2$ & $7$& $98$ & $3$ & $3$ & $7$& $98$ & $3$ & $4$ & $7$& $98$ & $3$ & $5$ & $7$& $98$ & $3$ & $6$ & $7$ \\ \hline
$98$ & $3$ & $7$ & $97$& $98$ & $3$ & $8$ & $97$& $98$ & $3$ & $9$ & $97$& $350$ & $3$ & $1$ & $5$& $350$ & $3$ & $2$ & $5$& $350$ & $3$ & $3$ & $5$ \\ \hline
$350$ & $3$ & $4$ & $5$& $350$ & $3$ & $5$ & $7$& $350$ & $3$ & $6$ & $7$& $350$ & $3$ & $7$ & $29$& $350$ & $3$ & $8$ & $29$& $350$ & $3$ & $9$ & $29$ \\ \hline
$24$ & $4$ & $1$ & $7$& $24$ & $4$ & $2$ & $7$& $24$ & $4$ & $3$ & $7$& $24$ & $4$ & $4$ & $11$& $24$ & $4$ & $5$ & $11$& $24$ & $4$ & $6$ & $11$ \\ \hline
$24$ & $4$ & $7$ & $11$& $24$ & $4$ & $8$ & $11$& $24$ & $4$ & $9$ & $23$& $25$ & $4$ & $1$ & $5$& $25$ & $4$ & $2$ & $5$& $25$ & $4$ & $3$ & $5$ \\ \hline
$25$ & $4$ & $4$ & $5$& $25$ & $4$ & $5$ & $11$& $25$ & $4$ & $6$ & $11$& $25$ & $4$ & $7$ & $11$& $25$ & $4$ & $8$ & $23$& $25$ & $4$ & $9$ & $23$ \\ \hline
$32$ & $4$ & $1$ & $5$& $32$ & $4$ & $2$ & $5$& $32$ & $4$ & $3$ & $29$& $32$ & $4$ & $4$ & $29$& $32$ & $4$ & $5$ & $29$& $32$ & $4$ & $6$ & $29$ \\ \hline
$32$ & $4$ & $7$ & $29$& $32$ & $4$ & $8$ & $29$& $32$ & $4$ & $9$ & $29$& $33$ & $4$ & $1$ & $5$& $33$ & $4$ & $2$ & $11$& $33$ & $4$ & $3$ & $11$ \\ \hline
$33$ & $4$ & $4$ & $11$& $33$ & $4$ & $5$ & $11$& $33$ & $4$ & $6$ & $11$& $33$ & $4$ & $7$ & $11$& $33$ & $4$ & $8$ & $11$& $33$ & $4$ & $9$ & $11$ \\ \hline
$48$ & $4$ & $1$ & $5$& $48$ & $4$ & $2$ & $23$& $48$ & $4$ & $3$ & $23$& $48$ & $4$ & $4$ & $23$& $48$ & $4$ & $5$ & $23$& $48$ & $4$ & $6$ & $23$ \\ \hline
$48$ & $4$ & $7$ & $23$& $48$ & $4$ & $8$ & $23$& $48$ & $4$ & $9$ & $23$& $49$ & $4$ & $1$ & $7$& $49$ & $4$ & $2$ & $7$& $49$ & $4$ & $3$ & $7$ \\ \hline
$49$ & $4$ & $4$ & $7$& $49$ & $4$ & $5$ & $7$& $49$ & $4$ & $6$ & $7$& $49$ & $4$ & $7$ & $23$& $49$ & $4$ & $8$ & $23$& $49$ & $4$ & $9$ & $23$ \\ \hline
$63$ & $4$ & $1$ & $5$& $63$ & $4$ & $2$ & $7$& $63$ & $4$ & $3$ & $7$& $63$ & $4$ & $4$ & $7$& $63$ & $4$ & $5$ & $7$& $63$ & $4$ & $6$ & $7$ \\ \hline
$63$ & $4$ & $7$ & $31$& $63$ & $4$ & $8$ & $31$& $63$ & $4$ & $9$ & $31$& 
$24$ & $5$ & $1$ & $7$& $24$ & $5$ & $2$ & $7$& $24$ & $5$ & $3$ & $7$ \\ \hline
$24$ & $5$ & $4$ & $11$& $24$ & $5$ & $5$ & $11$& $24$ & $5$ & $6$ & $11$& 
$24$ & $5$ & $7$ & $11$& $24$ & $5$ & $8$ & $11$& $24$ & $5$ & $9$ & $23$\\ \hline $32$ & $5$ & $1$ & $7$& $32$ & $5$ & $2$ & $7$& $32$ & $5$ & $3$ & $29$ &
$32$ & $5$ & $4$ & $29$& $32$ & $5$ & $5$ & $29$& $32$ & $5$ & $6$ & $29$\\ \hline $32$ & $5$ & $7$ & $29$& $32$ & $5$ & $8$ & $29$& $32$ & $5$ & $9$ & $29$& 
$48$ & $5$ & $1$ & $11$& $48$ & $5$ & $2$ & $11$& $48$ & $5$ & $3$ & $11$\\ \hline $48$ & $5$ & $4$ & $11$& $48$ & $5$ & $5$ & $11$& $48$ & $5$ & $6$ & $11$  &
$48$ & $5$ & $7$ & $23$& $48$ & $5$ & $8$ & $23$& $48$ & $5$ & $9$ & $23$\\ \hline 
$29$ & $7$ & $1$ & $13$& $29$ & $7$ & $2$ & $13$& $29$ & $7$ & $3$ & $13$& $29$ & $7$ & $4$ & $13$& $29$ & $7$ & $5$ & $13$& $29$ & $7$ & $6$ & $13$ \\ \hline
$29$ & $7$ & $7$ & $13$& $29$ & $7$ & $8$ & $13$& $29$ & $7$ & $9$ & $13$& $30$ & $7$ & $1$ & $13$& $30$ & $7$ & $2$ & $13$& $30$ & $7$ & $3$ & $13$ \\ \hline
$30$ & $7$ & $4$ & $13$& $30$ & $7$ & $5$ & $13$& $30$ & $7$ & $6$ & $13$& $30$ & $7$ & $7$ & $13$& $30$ & $7$ & $8$ & $13$& $30$ & $7$ & $9$ & $29$ \\ 
\end{longtable}
\end{center} 
\end{proof}
The proof of \lemref{lem_s>9_n>4k} depends on combining \lemref{lem_1} and \lemref{lem_3}. This will be the basis of several proofs in this paper. We shall refer this argument as \textit{combining \lemref{lem_1} and \lemref{lem_3}}. We have given above all the details required for combining \lemref{lem_1} and \lemref{lem_3}. We may skip these details at several places in future and leave them to the reader.

\begin{lem}\label{lem2k_le_n_le_4k}
Assume that $G(x)$ has a factor of degree $k \geq 2$. Let $2k \leq n \leq 4k$ and $1 \leq s \leq 92$. Then
\begin{align*}
& (n,k,s) \in  \{ (4,2,7),(4,2,23),   (9,3,47), (10,5,4) \nonumber\}.
\end{align*}
\end{lem}

\begin{proof}
Let $3k \leq n \leq 4k$ with $k > 92$. Then by \lemref{lem_max_n}, we have a prime $p > 2 k $ dividing  $n-i$ for some $0 \leq i \leq k-1$. Then we derive from \lemref{lem_1}  that $2k < p < k+s$ which implies that
$$
92 < k < s \leq 92
$$
and this is a contradiction. Let $k \leq 92$. Let $S_1$ be the set consisting of triples $(n,k,s)$ with $1 \leq s \leq 92$ satisfying \lemref{lem_1}. We see that the cardinality of $S_1$  is $148$.
Further we check that every pair in $S_1$ satisfies \lemref{lem_3} except for $(n, k,s) \in S_2$ where
\begin{align}\label{eq_excep_1}
S_2= \{& (6,2,4),(6,2,14),(6,2,79),(8,2,6),(8,2,13),(8,2,48),(8,2,62),(9,3,5),(9,3,6), \\& (9,3,13),  (9,3,47), (16,4,12), (16,4,62) \nonumber\}.
\end{align}
\noindent
Thus $S_2$ is obtained by  combining \lemref{lem_1} and \lemref{lem_3}.  We give the  details of applying \lemref{lem_3} for $(9,3,12)$ as follows. The details for other pairs are similar.  Consider 
\begin{equation}\label{eq_cj}
g_1(x) = n!g(x) = \sum\limits_{j=0}^{n}  \frac{n!}{j!} a_j x^j = \sum\limits_{j=0}^{n}  c_j x^j  \quad \text{where} \quad c_j = \frac{n!}{j!} a_j \in \mbb Z.
\end{equation}
We check that \lemref{lem_3} holds with $h(x) = g_1(x)$ and $l=1$  and a suitable prime $p$ for each pair in $S_2$. 
\begin{align*}
c_0  & =2^8 * 3^4 * 5^2 * 7^2 * 13 * 17 * 19, 
\hspace*{5 mm}
c_1  = 2^8 * 3^5 * 5^2 * 7 * 13 * 17 * 19, \\
c_2 & = 2^8 * 3^5 * 5 * 7 * 13 * 17 * 19,
\hspace*{5 mm}
c_3  = 2^8 * 3^4 * 5 * 7^2 * 13 * 17,\\
c_4 & =2^6 * 3^3 * 5 * 7^2 * 13 * 17,
\hspace*{5 mm}
c_5  = 2^6 * 3^3 * 5 * 7^2 * 13,
\hspace*{5 mm}
c_6  =2^3 * 3^2 * 5 * 7^2 * 13,\\
c_7 & =2^3 * 3^2 * 7 * 13,
\hspace*{5 mm}
c_8  =  3^2 *13
\hspace*{5 mm}
c_9  = 1.
\end{align*}
The primes $q$ which divides $c_j$ for $0 \leq j \leq 6$ are $2, 3,5,7$ and $13$. Now we consider the Newton polygon with respect to primes $2, 3, 5, 7$ and $13$ which are given by
\begin{align*}
NP_{2}(f) & = \{ (0, 0), (1,0),  (9,8) \} 
\hspace*{5 mm}
NP_{3}(f) = \{ (0, 0),  (9,4) \},
\hspace*{5 mm}
NP_{13}(f)  = \{ (0, 0),  (9, 1) \}, \\
NP_{5}(f) & = \{ (0, 0), (2,0), (7, 1),  (9, 2) \} 
\hspace*{5 mm}
NP_{7}(f) = \{ (0, 0), (1, 0), (8,1), (9, 2) \}.
\end{align*}
We see that the right most edge has slope less than $1/3$ for $q =13$.  \lemref{lem_3} is also applied in the proof of last lemma.

\smallskip
Let $2.5k \leq n < 3k$. 
First we argue when every term in $n(n-1)\dots (n-k+1)$ is a composite number.  We consider odd terms in $n(n-1)\dots (n-k+1)$ and denote by $\Delta$ the product of these odd terms. They satisfy
$$
\substack{\text{ Number of odd terms in} \atop (n-k+1)\dots (n-1)n } \geq  \begin{cases}
                  \frac{k}{2} &\text{if $k$ is  even ,}\\
                  \frac{k-1}{2} &\text{if $k$ is odd}.
                  \end{cases}
$$
Further each can be written as $n-\delta -2j$ where $\delta \in \{ 0,1 \}$ according as $n$ odd or even, respectively, and $j \geq 0$ is an integer. 
Each term can be written as $ pq$  where $p \geq 3$ is the least prime dividing $n-\delta - 2 j$ and $q$ is an integer greater than or equal to $p$.  Then $q < k$ since $n < 3k$. We have $ n - k+1   \geq  2.5 k - k +1 = 1.5 k + 1 >  1.5 k> k$. Now we  apply \lemref{new_lem_lash4} for driving
\begin{equation*}
P \big( \Delta \big) > 2.\frac{(k-1)}{2} = k-1.
\end{equation*}
Thus $ P \Big( (n - \delta-2j_0) \Big)  \geq k $ for some  term which is a contradiction to $q < k$.  Now we may assume that  one of the terms  in $n(n-1)\dots (n-k+1)$ is a prime number $p_0$. 
Then by \lemref{lem_1}, $p_0$ divides $_s(n+s)$ as $p_0 > 1.5 k > s$ for $k \geq 62$. Thus  we get, $1.5 k < p_0 < s + k$ for $k \geq 62$.  Then $0.5 k < s \leq 92$ which  gives $ k < 184$.  By combining \lemref{lem_1} and \lemref{lem_3},  we find that  $G(x)$ does not have a factor of degree $2 \leq k < 184$ for all  $2.5 k \leq n < 3k$ except for $(n,k,s) \in S_3$ where 
\begin{equation}\label{eq_excep_2}
 S_3 = \{ ( 8,3, 6), (8,3,13) \}.
\end{equation}

Let $n \in [2k +93,  2.5 k)$. Then $k > 186$.   We put 
$$
\theta =  \begin{cases}
                  \frac{1}{2} &\text{if $k$ is odd,}\\
                  1 &\text{if $k$ is even}
                  \end{cases}
$$
and $K : = 1.5k + \theta > 279 $ is an integer satisfying $n - k +1 \leq 2.5 k - \theta-k +1 = K$.
Then $[n-k+1, n] \supseteq [K, 2k+93] \supseteq [K, \frac{2}{1.5}K] \supseteq [K, 1.064286 K]$ and a prime number $q_0 \in  [K, 1.064286 K]$ by \lemref{lem_pnt}(b). Thus there exists a prime $q_0 \in [n-k+1, n] $ for $k > 186$. Since $q_0 \geq n -k +1 \geq 2k +93 -k +1 = k+94 > k$ and $k+94  > 280> s$ since $k > 186$ and $s \leq 92$, we derive from \lemref{lem_1} that $q_0 | (n+1)\dots (n+s)$. Thus $k + 94 \leq q_0 < s+k \leq 92+k$ and this is a contradiction. 

Let $n \in [2k, 2k +92]$ with $k > 105$. Then we have $n - k+1 \leq 2k +92 -k+1 = k+93 $ and $2k \leq n$. Therefore the interval $[k +93, 2k]$  is contained in $[n-k+1, n]$. By \lemref{lem_prime}(b)  there exists a prime $p_0$  in $[k +93, 2k]$.  This implies that $p_0 =n -i$ for some $0 \leq i \leq k-1$. Further $p_0 \geq n-k+1 \geq 2k-k+1 = k+1 > 106 > s$.  Now by \lemref{lem_1}, $p_0$ divides  $p_0 | {}_s(n+s)$ since $p_0 > s$. Thus we get $k+93 \leq p_0 < s+k $ which is a contradiction. Now consider $n \in [2k, 2k +92]$ with $k \leq 105$. We have $2k +92 \leq 4k$ which implies that $k \geq 46$.
We have $n \in [2k, 2k +92]$ with $46 \leq k \leq 105$ and for $2 \leq k \leq 45$ we have $n \in [2k,4k]$.  By combining \lemref{lem_1} and  \lemref{lem_3} we conclude that $G(x)$ does not have a factor of degree $k$ except for $(n, k, s) \in S_4$ where
\begin{align}\label{eq_excep_3}
S_4 & =  \{ (4,2,2),(4,2,7),(4,2,23),(4,2,47),(6,2,4),(6,2,14),(6,2,79),
(8,2,6),(8,2,13),  \\ & \nonumber  (8,2,48),(8,2,62),(6,3,47),(8,3,6),(8,3,13),
(9,3,5),(9,3,6),  (9,3,13),(9,3,47),  \\ & \nonumber (9,4,5),(9,4,6),(16,4,12),(16,4,62),(10,5,4),(10,5,5)  \nonumber\}.
\end{align} 
\noindent
Combining \eqref{eq_excep_1}, \eqref{eq_excep_2} and \eqref{eq_excep_3} and deleting triples given by \lemref{lem_dum_exp}, we get a complete set of exceptions stated in \lemref{lem2k_le_n_le_4k}.
\end{proof}

\begin{lem}\label{lem_s>9}
Assume that $G(x)$ has a factor of degree $k \geq 2$. Then $s > 9$ unless
\begin{align*}\label{exceptions_1}
 (n, k, s ) \in\{& (4,2,7),(10,5,4) \nonumber\}.
\end{align*}
\end{lem}

\begin{proof}
Assume that $s \leq 9$.  Suppose that $G(x)$  has a factor of degree $k \geq 2$. Let  $ n > 4k $. By \lemref{lem_s>9_n>4k} we may assume that $(n,k) \not\in T$. Let $k=2$. By \lemref{new_lem_1}, we have 
$$
p = P({}_k(n)) > 4.42 k  \hspace*{.5cm} \text{for} \hspace*{.5cm} n > 4k.
$$
Therefore $p \geq 11$ and $p | (n-i)$ for some $i$ with $0 \leq i \leq k-1$. By \lemref{lem_1}, $p | (n+j)$ for some $j$ with $1\leq j \leq s$ since $p \geq 11 > s$. Then
$$
11 \leq p \leq i+j < s+k \leq 9+ 2
$$
which is a contradiction. Let $k >2$. Then $p > 4.42 k$ by \lemref{new_lem_1} and as above 
$$
4.42 k < p \leq i+j < s+k 
$$
which implies that 
\begin{equation}\label{eq_new_3.42k}
3.42 k < s.
\end{equation}
This is a contradiction since $s \leq 9$.
If $2k \leq n \leq 4k$, the assertion follows from \lemref{lem2k_le_n_le_4k} and $s \in [1,9]$.
\end{proof}

\begin{lem}\label{lem_n<100}
Assume that $G(x)$ has a factor of degree $k \geq 2$.
\begin{itemize}
\item[(a)] Let $2k \leq n \leq 300$ with $2 \leq k \leq 26$. Let $10 \leq s \leq 92$ and $P({}_k(n)) < 100$. Then we have the following triples $(n,k,s)$
\begin{align*}
 \{(4,2,23), (9,2,19), (9,2,47), (16, 2, 14), (16, 2, 34), (16, 2, 89), (9, 3, 47), (16, 3, 19)  \}.
\end{align*}

\item[(b)] Let $n > 300$, $k \in [2,26]$, $P({}_k(n)) < 100$ and $10 \leq s \leq 92$.  Then we have no triples.
\end{itemize}
\end{lem}

\begin{proof}
\begin{itemize} 
\item[(a)]  By combining \lemref{lem_1} and \lemref{lem_3} with $10 \leq s \leq 92$ and $P({}_k(n)) < 100$,  we get $\Omega_1$ where
\begin{align*}
\Omega_1 = \{ & (4,2,23),(4,2,47),(6,2,14), (6,2,79), (8,2,13),(8,2,48),(8,2,62),  (9,2,19),(9,2,26), \\ &  (9,2,47), (12,2,43), (16,2,14),(16,2,19),(16,2,24),(16,2,34), (16,2,79), (16,2,89),     \nonumber\\ &  (18,2,16),  (24,2,22),  (32,2,30),   (48,2,46),(54,2,52),(64,2,62),  (72,2,70),  (6,3,47), \nonumber\\ & (8,3,13),(9,3,13),  (9,3,47), (16,3,12), (16,3,13), (16,3,19), (16,3,62), (18,3,33), \nonumber\\ & (36,3,15),  (16,4,12),  (16,4,62),  (28,4,24),(81,4,77)\nonumber \}.
\end{align*}
Then $\Omega_1 \setminus \Omega$ is the triples of \lemref{lem_n<100} (a). 

\item[(b)] 
Define
\begin{equation*}
D_k : =\{ n> 300 | P({}_k(n)) < \min\{s+k, 100\} \}. 
\end{equation*}
 Let $2 \leq k \leq 26$ and $10 \leq s \leq 92$. For the solutions of  $P({}_2(n)) < 100$, we use the Table given in \cite{LN}. Since
$$
4n(n-1)= 4\Big\{ \big(n-\frac{1}{2} \big)^2-\frac{1}{4} \Big\} = (2n-1)^2-1,
$$
it is enough to consider the odd solutions of $x^2 -1$ attached in the \textit{Homepage of Filip Najman/ Publications }.  By \lemref{lem_new_4}, we have $k \leq 8$  and $D_k = \emptyset$ for $ k > 8$. We have $D_2$ for $n > 300$ obtained from all the odd solutions of $x^2 - 1$, see \cite{LN}.  Further we have
 \begin{align*}
 D_3=\{& 36519,17577,17459,6785,5831,5291,4901,3887,3871,3551,3485,3479,2925,\\ & 2279,2211,2015,1887,1769,1681,1519,1421,1377,1333,1275,1273,1241,1221,\\ & 1161,1025,989,945,901,869,847,805,783,781,737,731,639,611,551,533,531,\\ &529,497,495,475,425,407,377,371,365,343 \} \\
 D_4=\{& 5831,4901,3479,2211,1333,1275,1221,1025,989,783,731,639,611,533,531,\\ & 497,495,475,377,371,343 \}\\
 D_5= \{ & 1275,783,533,531,497 \}, 
 D_6= \{  4901,533,497 \}, D_7= \{ 533\}, D_8 = \emptyset.
 \end{align*}
 By combining \lemref{lem_1} and \lemref{lem_3} and also using \lemref{lem_s>p^r} for $k=2$, we get no factor of degree $k$. We can see for $(n,k,s) = (332110803172167361,2,92)$. We have the factorisation of $n= 2521 * 187177 * 703813633$. By \lemref{lem_s>p^r}, we get a contradiction for $p = 2521$. Consider  $(n,k,s) = (3939649,2,92)$. We have $n = 7^2 * 37 * 41 * 53$ whereas $n-1 = 2^6 * 3 * 17^2 * 71$.  By \lemref{lem_s>p^r},  get a contradiction for $p = 17$. The exceptions $(n,k,s) \in (134849,2,s)$ for $86 \leq s \leq 92$ are excluded by \lemref{lem_1} with $p=41$.
\end{itemize}
\end{proof}

\section{\bf{Proof of \thmref{thm_ShSi} }  }
\begin{enumerate}

\item[(a)]  Assume that the assumptions of \thmref{thm_ShSi} are satisfied and let $(n,k,s) \neq (10,5,4)$. We may assume that $(n,k,s) \neq (4,2,7)$ otherwise the assertion follows.
Let $n > 4k$ and  $(n,k) \not\in T$.  We have $p_0 = P({_k}(n)) > 4.42 k > s$ by \lemref{new_lem_1}.  By \lemref{lem_1}, we have $p_0 | {}_s(n+s)$. As in \eqref{eq_new_3.42k}, we get $3.42k < s$ which  contradicts $s \leq 1.9 k $. Let $(n,k) \in T$.  Then  we have
$$
(n,k) \in \{ (29,7), (30,7) \} \quad \text{with} \quad 10 \leq s \leq 12
$$
since $k \geq 6$. In the following Table \ref{tab;table_k=7} we have included a prime $p$ with which \lemref{lem_1} does not hold.

\begin{table}[ ! h]
  \centering
  \caption{ For $ (n, k) \in  \{(29,7), (30,7) \}$  not satisfying \lemref{lem_1}. }
\label{tab;table_k=7}
 \begin{adjustbox}{width=\textwidth}
\begin{tabular}{|c|c|c||c|c|c||c|c|c||c|c|c||c|c|c||c|c|c|}
\hline
$n$ &  $s$ & $p$ &$n$ &   $s$ & $p$ & $n$ &  $s$ & $p$ &$n$ &  $s$ & $p$ &  $n$ & $s$ & $p$&  $n$ &  $s$ & $p$ \\
\hline
$29$ & $10$ & $23$& $29$ & $11$ & $23$& $29$ & $12$ & $23$& $30$ & $10$ & $29$& $30$ &  $11$ & $29$& $30$ &  $12$ & $29$ \\ \hline
\end{tabular}
\end{adjustbox}
\end{table}

It remains to consider $2k \leq n \leq 4k$.  By applying  \lemref{lem_ShSi}  with $\phi = \frac{1}{40}$ and \lemref{lem_prime}(a) with  $\theta =\frac{1}{39}$, we get
$$
s > 2 \Big(1-\frac{2}{40}  \Big)k  = 1.9  k \quad \text{for} \quad n \geq 821
$$
which is a contradiction. Now we have $n \leq 820$ with $2k \leq n \leq 4k$.  For $6 \leq k \leq 410$, we have $2k \leq n \leq \min \{820, 4k \}$. 
For a fixed $k_0$ with $6 \leq k_0 \leq 410$,  we have a range for $n \in [2k_0, 4k_0]$ with $n \leq \min \{820, 4k_0 \}$. Now for fixed $n_0 \in [2k_0, 4k_0]$ with $n_0 \leq \min \{820, 4k_0 \}$, we calculate the prime factors of $n_0(n_0-1) \dots (n_0-k_0+1)$ and the prime factors of  $\frac{(n_0+1) \dots (n_0+s)}{s!} $ for each $s \in [10, 1.9 k_0]$. If there exist a prime $p_0$ such that $p_0 | n_0(n_0-1) \dots (n_0-k_0+1)$ but $p_{0} \nmid \frac{(n_0+1) \dots (n_0+s_0)}{s_0!} $ for  some $s_0$, then by \lemref{lem_1},  $G(x)$ with respect to the given $n_0$ and $s_0$ does not have a factor of degree $k_0$ otherwise we apply \lemref{lem_3} to $(n_0, k_0, s_0)$. For a fixed $k \geq 6$, we check that $G(x)$ for any $ n$ and any $s$ in their respective ranges, does not have a factor of degree $k$.

\item[(b)] By \lemref{lem_ShSi} with $\phi = \frac{1}{1001}$ and \lemref{lem_prime}(a) with $ \theta = \frac{1}{1000}$, we get
$$
s > 2 \Big(1-\frac{2}{1001}  \Big)k  = \frac{1998}{1001}k >  1.99  k \quad \text{for} \quad n \geq  48732, n \leq 4k.
$$
Let $n > 4k$. Let $(n,k) \not \in T$.  Assume that $s \leq 2k$. By \eqref{eq_new_3.42k}, we get $3.42k < s$ which is a contradiction to $s \leq 2k$. Thus $s > 2k$ and this gives \eqref{eq_new_s}. Let $(n,k) \in T$. Then $k \leq 7$ by \lemref{new_lem_1}. Assume that $1 \leq s \leq 2k$.  Then we check that $(n,k,s)$ does not satisfy \lemref{lem_1} for every $1 \leq s \leq 2k$. Therefore $s >2k$ and hence \eqref{eq_new_s} holds.
\end{enumerate}

\section{\bf{Proof of \thmref{thm_ShSi2_new} }}  
Let $s \leq 92$. Assume that $G(x)$ has a factor of degree $k \geq 2$.  By \lemref{lem_s>9}, it suffices to show  that $(n,k,s) \in \Omega_2 $ for $10 \leq s \leq 92$. We have $s > 1.9 k$ by \thmref{thm_ShSi}. Therefore  $ k < \frac{92}{1.9}$. This gives $k \leq 48$.

Let $n > 4k$  and $(n,k) \in T$. Then  $k \leq 7$ by \lemref{new_lem_1}.  By combining \lemref{lem_1} and \lemref{lem_3}, we check that $G(x)$ does not have a factor of degree $k$ for all pairs in $T$.  Let $(n,k) \not \in T$. Now by \lemref{new_lem_1} and \lemref{lem_1}, we have 
\begin{equation*}\label{eq_thm3}
4.42k < s+k \leq  92 + k
\end{equation*}
which implies that $ k \leq 26$. Thus  we have $2 \leq k \leq 26$ with $p=P({}_k(n)) < s+k$  by \lemref{lem_1} which implies that $p < 100$ since $s \leq 92$ and $k \leq 9$ by \lemref{lem_new_4}. Now, we see that the assumption of \lemref{lem_n<100} are satisfied and hence the assertion for $n > 4k$ follows.

If $2k \leq n \leq 4k$, The assertion follows by \lemref{lem2k_le_n_le_4k}.


\begin{thebibliography}{22}

\bibitem{Dum}  Dumas G., Sur quelques cas d{'}irr{\'e}ductibilit{\'e} des poly{\^o}mes \'{a} coefficients rationnels,  Journal de Math. Pure et Appl., {\bf 2}, 191-258 (1906).

\bibitem{Dus}  Dusart P., In{\'e}qualitie{\'e}s explicites pour $\psi (x)$, $\theta(x)$, $\Pi(x)$ et les nombres premier,  C. R. Math. Acad. Sci. Soc. R. Can., {\bf 21} , 53-59 (1999).

\bibitem{Fila} Filaseta M., The irreducibility of all but finitely many Bessel polynomials,  Acta
Math., {\bf 174}, 383–397 (1995).

\bibitem{FFL} Filaseta M.,  Finch Carrie, Leidy, J. Russell , T. N. Shorey's influence in the theory of irreducible polynomials. Diophantine equations,   Tata Inst. Fund. Res. Stud. Math., Tata Inst. Fund. Res., Mumbai, {\bf 20},  77-102 (2008).

\bibitem{Haj}  Hajir F., Some An-extensions obtained from generalized Laguerre polynomials,  J. Number Theory, {\bf 50}, 206–212 (1995).

\bibitem{Haj1}  Hajir F., Algebraic properties of a family of generalized Laguerre polynomials,   Canad. J. Math., {\bf 61}, 583–603 (2009).


\bibitem{HeiKem} Heiko Harborth and Arnfried Kemnitz, Calculation of Bertrand's postulate, Mathematics Magazine, {\bf 54:1}, 33-34 (1981). 

\bibitem{JiLaSa}  Jindal A, Laishram S. and Sarma R., Irreducibility and Galois groups of generalized Laguerre polynomials $L_n^{(-1-n-r)}(x)$,  J. Number Theory, {\bf 183}, 388-406 (2018). 

\bibitem{LaSh_new}  Laishram S. and Shorey T N, Number of prime divisors in a product of terms of an arithmetic progression,  Indagationes Mathematicae, 15{\bf 4}, 505-521 (2004). 

\bibitem{LaSh}  Laishram S. and Shorey T N, The greatest prime divisor of a product of consecutive integers,  ibid, {\bf 120}, 299-306 (2005). 

\bibitem{LaNaSh} Laishram Shanta, Nair Saranya G. and Shorey, T. N.,  Irreducibility of generalised Laguerre polynomials $L_n^{(\frac{1}{2}+u)}(x)$ with integer $u$,  J. Number
Theory, {\bf 160}, 76-107 (2016). 

\bibitem{Leh}  Lehmer D. H., On a problem of St\"{o}rmer,  Illinois J. Math, {\bf 8}, 57-79 (1964).

\bibitem{LN} Luca Florian and Najman Filip . On the largest prime factor of $x^2-1$,  Mathematics of Computation, {\bf 80}, number 273, 429-435 (2010).  Table for odd solutions attached in the Home page of Najman/Publications.

\bibitem{NaSh} Nair Saranya G. and Shorey, T. N.,  Lower bounds for the greatest prime factor of product of consecutive positive integers,  J. Number
Theory, {\bf 159}, 307-328 (2016). 

\bibitem{NaSh1} Nair Saranya G. and Shorey, T. N.,  Irreducibility of Laguerre polynomial $L_n^{(-n-s-1)}(x)$,  Indagationes Mathematicae, {\bf 26}, 615-625 (2015). 

\bibitem{NaSh2} Nair Saranya G. and Shorey, T. N.,  Generalised Laguerre polynomials with applications,  The Mathematics Student, {\bf 86}, 87-101 (2017)

\bibitem{Sell}  Sell E. A., On a certain family of generalized Laguerre polynomials,  J. Number
Theory, {\bf 107}, 266-281 (2004). 

\bibitem{Syl}   Sylvester J. J., On arithmetic series,  Messenger of Mathematics, {\bf XXI},  pp. 1-19, 87-120 (1892) and Mathematical Papers, {\bf 4},  pp. 687-731 (1912). 

\bibitem{TijSh1}  Shorey T. N. and  Tijdeman R., Generalizations of some irreducibility results by Schur,  Acta Arith., {\bf 145}, 341–371 (2010).

\bibitem{Sch}  Schur I.,  Einige Satze uber Primzahlen mit Anwendungen auf Irreduzibilita tsfragenI,  Sitzungsber. Preuss. Akad. Berlin Phys. Math. Kl., {\bf 14}, 125–136 (1929).



\end{thebibliography}
\end{document}